\documentclass[10.5pt]{amsart}
\usepackage[english]{babel}
\usepackage{amsmath}
\usepackage{amsthm}
\usepackage{amssymb}
\usepackage{mathrsfs}
\usepackage{enumerate}
\usepackage[notcite, final, notref]{showkeys}
\usepackage{dsfont}

\usepackage[dvips]{color}
\newtheorem{theorem}{Theorem}[section]

\newtheorem{lemma}[theorem]{Lemma}
\newtheorem{proposition}[theorem]{Proposition}

\theoremstyle{definition}
\newtheorem{remark}[theorem]{Remark}

\usepackage{xcolor}

\title[Adaptive orthonormal systems for matrix-valued functions]{Adaptive orthonormal systems for matrix-valued functions}

\author[D. Alpay]{Daniel Alpay}
\address{(DA) Department of Mathematics\\
Ben-Gurion University of the Negev\\
Beer-Sheva 84105 Israel}
\email{dany@math.bgu.ac.il}

\author[F. Colombo]{Fabrizio Colombo}
\address{(FC) Politecnico di
Milano\\Dipartimento di Matematica\\Via E. Bonardi, 9\\20133 Milano,
Italy}
\email{fabrizio.colombo@polimi.it}

\author[T. Qian]{Tao Qian}
\address{(TQ) Department of Mathematics\\
Universityof Macau\\
Macau}
\email{fsttq@umac.mo}

\author[I. Sabadini]{Irene Sabadini}
\address{(IS) Politecnico di
Milano\\Dipartimento di Matematica\\Via E. Bonardi, 9\\20133 Milano\\Italy}
\email{irene.sabadini@polimi.it}

\thanks{
The authors thank
Macao Science and Technology Fund FDCT/098/2012/A3
and
Macao Science and Technology Fund FDCT/099/2014/A2.
}

\begin{document}
\maketitle
\begin{abstract}

In this paper we consider functions in the Hardy space $\mathbf{H}_2^{p\times q}$ defined in the unit disc of matrix-valued. We show that it is possible, as in the scalar case, to decompose those functions as linear combinations of suitably modified matrix-valued Blaschke product, in an adaptive way. The procedure is based on a generalization to the matrix-valued case of the maximum selection principle which involves not only selections of suitable points in the unit disc but also suitable orthogonal projections.  We show that the maximum selection principle gives rise to a convergent algorithm. Finally, we discuss the case of real-valued signals.

 \end{abstract}


\date{today}
\tableofcontents
\section{Introduction}
\setcounter{equation}{0}

Functions in the Hardy space  $\mathbf{H}_2(\mathbb D)$
of the open unit disc $\mathbb D$ can be decomposed
into linear combinations of functions which are modified Blaschke products
\begin{equation}\label{TMsystem}
B_n(z)= \frac{\sqrt{1-|a_n|^2}}{1-z\overline{a_n}} \prod_{k=1}^{n-1} \frac{z-a_k}{1-z\overline{a_k}}, \qquad n=1,2,\ldots
\end{equation}
where the points
$a_n\in\mathbb D$
are adaptively chosen according  to the
function to be decomposed, see \cite{QWa1}. It is important to note that these points do not necessarily satisfy the so-called hyperbolic non-separability condition
\begin{equation}\label{conditionTM}
\sum_{n=1}^\infty 1-|a_n|=\infty.
\end{equation}
The  system \eqref{TMsystem}, which is orthonormal, is called Takenaka--Malmquist
system.
It is a basis of the Hardy space $\mathbf{H}_2(\mathbb D)$ and, more in general, of $\mathbf{H}_p(\mathbb D)$, $1\leq p\leq\infty$, if and only if  \eqref{conditionTM} is satisfied.
It is possible to show, see \cite{QWa1}, that the points $a_n$ can be chosen  to decompose a given  function
$f$ into basic functions, each of which has nonnegative analytic instantaneous frequency. A system \eqref{TMsystem} satisfying this property is called an
adaptive rational orthonormal system. A signal that possesses a nonnegative analytic instantaneous frequency function is said to be mono-component. It can be real or complex-valued. If, in particular, taking $a_1=0,$ then the boundary values of the modified Blaschke products $B_n$ are mono-component for all $n\in\mathbb N$.
We note that the condition \eqref{conditionTM} is not necessarily satisfied, and so the system is not necessarily complete
in $\mathbf{H}_2(\mathbb D)$. However, the convergence to $f$ is fast.
\\
As we said, the adaptive decomposition is designed in order to obtain a decomposition of a functions into mono-component signals. This method has been intensively studied in the past few years, see \cite{WQW2012,Qian_106,QWa1,Qian2014,QLS2013,QWP2009}. It gives rise to an algorithm which is a variation of the greedy algorithm, see \cite{MZ, Tem, QWo2015}.
\\
An algorithm to perform an adaptive decomposition, given $f$, can be assigned as follows.
One considers a so-called dictionary $\mathcal D,$ being a family of elements of unit norm whose span is dense in the Hilbert space $\mathcal H.$ Given $f\in\mathcal H$ we select
$u_1, \ldots , u_n\in\mathcal D$ such that
\[
f=\sum_{k=1}^\infty\langle f_k ,u_k\rangle u_k
\]
where the functions $f_k$ are defined inductively, starting from $f_1=f$ and setting
\[
f_k=f-\sum_{\ell=1}^{k-1} \langle f_\ell, u_\ell\rangle u_\ell,
\]
where $\langle\cdot, \cdot\rangle$ denotes the inner product in $\mathcal H$.\\
In the paper \cite{QWa1} $\mathcal H=\mathbf H_2(\mathbb D)$ with its standard inner product, the dictionary consists of
the normalized Szeg\"o kernels,
\[
\mathcal D=\left\{e_a(z) =\dfrac{\sqrt{1-|a|^2}}{1-z\overline{a}}, \ \ a\in\mathbb D\right\}.
\]
Note that the reproducing kernel property of $e_a$ in $\mathbf H_2(\mathbb D)$ yields
\[
\langle f, e_a\rangle = \sqrt{1-|a|^2}f(a).
\]
Let $f\in\mathbf H_2(\mathbb D)$ and set $f_1=f$. For any $a_1\in\mathbb D$
\begin{equation}\label{twosum}
f(z)= \langle f_1, e_{a_1}\rangle e_{a_1}(z) +f_2(z)\frac{z-a_1}{1-z\overline{a_1}}
\end{equation}
where
\[
f_2(z)=\frac{f_1(z)-\langle f_1, e_{a_1}\rangle e_{a_1}(z)}{\frac{z-a_1}{1-z\overline{a_1}}}.
\]
One can show that $f_2\in \mathbf H_2(\mathbb D)$ and so the procedure can be repeated.
The transformation from $f_1$ to $f_2$ is called generalized backward-shift.
The two summands in \eqref{twosum} are orthogonal, thus
\[
\|f\|^2= |\langle f_1, e_{a_1}\rangle|^2 +\|f_2\|^2.
\]
The maximal selection principle asserts that it is possible to choose $a_1\in\mathbb D$ such that
\[
a_1=\max\{|\langle f_1, e_{a}\rangle|^2=(1-|a|^2) |f_1(a)|^2,\ \ a\in\mathbb D\}.
\]
The procedure can be iterated and after $n$ steps one has
\[
f(z)=\sum_{k=1}^n \langle f_k, e_{a_k}\rangle B_k(z) +f_{n+1}(z) \prod_{k=1}^n \dfrac{z-a_k}{1-z\overline{a_k}},
\]
where
\[
a_k=\max\{|\langle f_k, e_{a}\rangle|^2=(1-|a|^2) |f_k(a)|^2,\ \ a\in\mathbb D\}, \quad k=1,\ldots, n
\]
and
\begin{equation}
\label{red0}
f_k(z)=\frac{f_{k-1}(z)-\langle f_{k-1}, e_{a_{k-1}}\rangle e_{a_{k-1}}(z)}{\frac{z-a_{k-1}}{1-z\overline{a_{k-1}}}}.
\end{equation}
The function $f_k$ is called the $k$-th reduced remainder (see \cite[(11) p. 850]{Qian_106}). Its matrix-valued counterpart is given by \eqref{reduced2}. One can easily show the relations
\begin{equation}\label{equivalence}
\langle f_k, e_{a_k}\rangle = \langle g_k, B_k \rangle =\langle f, B_k \rangle ,\end{equation}
where $g_k$ is the $k$-th standard remainder, defined through
\[ g_k= f-\sum_{l=1}^{k-1} \langle f, B_l \rangle B_l.\]

As before, the orthogonality of the summands and the fact that $B_k$ is unimodular on $\partial\mathbb D$, give
\[
 \|f(z)-\sum_{k=1}^n \langle f_k, e_{a_k}\rangle B_k(z)\|^2= \|f(z)\|^2-\sum_{k=1}^n |\langle f_k, e_{a_k}\rangle |^2= \|f_{n+1}\|^2.
\]
Since it can be shown that $ \|f_{n+1}\|\to 0$ as $n\to\infty$ (see \cite[Theorem 2.2]{QWa1}), we have the  formula
\[
f(z)=\sum_{k=1}^\infty \langle f_k, e_{a_k}\rangle B_k(z)
\]
called adaptive Fourier decomposition, abbreviated as AFD.\smallskip

In this paper, we extend some of the results of \cite{QWa1} to the matrix-valued case. For $w\in\mathbb D$ we will use the notations $e_w$ and $b_w$ for the normalized Cauchy kernel and Blaschke factor at the point $w$ respectively, that is:
\begin{equation}\label{Szego and Moebius}
e_w(z)=\frac{\sqrt{1-|w|^2}}{1-z\overline{w}}\quad \text{and}\quad b_w(z)=\frac{z-w}{1-z\overline{w}}
\end{equation}
The Szeg\"o dictionary now consists of the $\mathbb C^{p\times p}$-valued functions $Pe_w$, where $w$ belongs to the open unit disk $\mathbb D$ and $P\in
\mathbb C^{p\times p}$ is any orthogonal projection, that satisfies $P=P^2=P^*$.

\begin{remark}{\rm
In view of the polydisk setting, the operator-valued case will be considered in a later publication (see  \cite{abr2, MR1839829, QWo2015} for an approach to the polydisk setting using operator-valued functions).}
\end{remark}

We denote by $\mathbf H_2^{p\times q}$ the space of $p\times q$ matrices with entries in $\mathbf H_2(\mathbb D)$. When $q=1$ we write $\mathbf H_2^{p}$ rather than $\mathbf H_2^{p\times q}$.\smallskip

A function $F\in\mathbf H_2^{p\times q}$ if and only if it can be written as
\begin{equation}
\label{appro1}
F(z)=\sum_{n=0}^\infty F_nz^n,
\end{equation}
where $F_\ell\in\mathbb C^{p\times q}$, $\ell=1,2,\ldots$, are such that
\begin{equation}
\label{appro2}
\sum_{n=0}^\infty {\rm Tr}~(F_n^*F_n)<\infty.
\end{equation}
Let $G$ be another element of $\mathbf H_2^{p\times q}$, with power series expansion $G(z)=\sum_{n=1}^\infty G_nz^n$ at the origin.
We set
\begin{equation}
\label{innerprod}
[F,G]=\sum_{n=0}^\infty G_n^*F_n\,\,\in\,\,\mathbb C^{q\times q}
\end{equation}
and
\[
\|F\|^2={\rm Tr}\, [F,F]=\sum_{n=0}^\infty {\rm Tr}~(F_n^*F_n).
\]
We note that \eqref{innerprod} can be rewritten as
\begin{equation}
\lim_{\substack{r \rightarrow 1\\ r\in(0,1)}}\frac{1}{2\pi}\int_0^{2\pi}G(re^{it})^*F(re^{it})dt
\label{innerprod2}
\end{equation}
and so we also have
\begin{equation}
\label{appro21}
\sum_{n=0}^\infty {\rm Tr}~(G_n^*F_n)=\lim_{\substack{r \rightarrow 1\\ r\in(0,1)}}\frac{1}{2\pi}\int_0^{2\pi}
{\rm Tr}\,G(re^{it})^*F(re^{it})dt.
\end{equation}

Most, if not all, the material of Sections \ref{sec3} and \ref{sec2} is classical. Some proofs are provided for the convenience of the reader.
We refer to \cite{ag2,ag} for a study of these using state space methods.\\

An important condition in the algorithm is whether $F$ is a cyclic vector for the backward shift operator $R_0$, that is, whether
the closed linear span $\mathcal M(F)$ of the functions
\[
R_0^nFX,\quad n=0,1,2,\ldots\,\,\, \text{and}\,\,\, X\in\mathbb C^{q\times q}
\]
is strictly included in $\mathbf H_2^{p\times q}$ or not.
\section{The maximum selection principle}
\setcounter{equation}{0}
In this section we show that the maximum selection principle holds also in the matrix valued case. It allows to adaptively choose a sequence of points together with orthogonal projections for any given function in the Hardy space. We note that this selection principle does not exclude the possibility that the obtained sequence of points contains elements repeating more than once.
\begin{proposition}
Let $k_0\in\left\{1,\ldots, p\right\}$, and let $F\in\mathbf H_2^{p\times q}$. There exists $w_0\in\mathbb D$ and
an orthogonal projection $P_0$ of rank $k_0$ such that
\begin{equation}
\label{toronto}
(1-|w_0|^2)\left({\rm Tr}~[P_0F(w_0),F(w_0)]\right)\,\,\,\text{is maximum.}
\end{equation}
\end{proposition}

\begin{proof}
We first recall that for $f\in\mathbf H_2(\mathbb D)$ (that is, $p=q=1$), with power series $f(z)=\sum_{n=0}^\infty f_nz^n$, and for $w\in\mathbb D$, we have
\begin{equation}
\label{ineq11}
\sqrt{1-|w|^2}|f(w)|=|[f,e_{w}]|\le \|f\|.
\end{equation}
Let $F=(f_{ij})\in\mathbf H_2^{p\times q}$, where the $f_{ij}\in\mathbf H_2(\mathbb D)$ ($i=1,\ldots, p$ and $j=1,\ldots q$), and $\xi\in\mathbb C^{k_0\times p}$ such that $\xi\xi^*=I_{k_0}$. Then,
\[
\begin{split}
{\rm Tr}~[\xi F(w),\xi F(w)]&={\rm Tr}~F(w)^*\xi^*\xi F(w)\\
&\le {\rm Tr}~F(w)^*F(w)\\
&=\sum_{i=1}^p\sum_{j=1}^q|f_{ij}(w)|^2.
\end{split}
\]
Using \eqref{ineq11} for every $f_{ij}$, we obtain
\begin{equation}
\label{macau1234}
(1-|w|^2)\left({\rm Tr}~[\xi F(w),\xi F(w)]\right)\le\sum_{i=1}^p\sum_{j=1}^q\|f_{ij}\|^2=\|F\|^2.
\end{equation}
Let $\epsilon>0$. In view of \eqref{appro1}-\eqref{appro2} there exists a $\mathbb C^{p\times q}$-valued polynomial $P$ such that $\|F-P\|\le \epsilon$.
We have
\[
\begin{split}
(1-|w|^2)\left({\rm Tr}~[\xi F(w),\xi F(w)]\right)&\\
&\hspace{-3cm}
=(1-|w|^2)\left({\rm Tr}~[\xi(F-P)(w)+\xi P(w),\xi(F-P)(w)+\xi P(w)]\right)\\
&\hspace{-3cm}=(1-|w|^2)\|\xi(F-P)(w)+\xi P(w)\|^2\\
&\hspace{-3cm}\le(1-|w|^2)\left(\|\xi(F-P)(w)\|+\|\xi P(w)\|\right)^2\\
&\hspace{-3cm}\le2(1-|w|^2)\|\xi(F-P)(w)\|^2+2(1-|w|^2)\|\xi P(w)\|^2\\
&\hspace{-3cm}\le 2\|F-P\|^2+2(1-|w|^2)\|\xi P(w)\|^2 \quad (\text{where we have used \eqref{macau1234}})\\
&\hspace{-3cm}\le 2\epsilon^2+2(1-|w|^2)\|P(w)\|^2 .
\end{split}
\]
This ends the proof since $(1-|w|^2)\|P(w)\|^2$ tends to $0$ as $w$ approaches the unit circle and since $\xi^*\xi$ is a rank $k_0$ orthogonal projection.
\end{proof}

We write
\begin{equation}
\label{bastille}
F(z)=P_0F(w_0)e_{w_0}(z)\sqrt{1-|w_0|^2}+F(z)-P_0F(w_0)e_{w_0}(z)\sqrt{1-|w_0|^2}.
\end{equation}
\begin{lemma}
Let
\[
\begin{split}
H(z)&=F(z)-P_0F(w_0)e_{w_0}(z)\sqrt{1-|w_0|^2}\\
H_0(z)&=P_0F(w_0)e_{w_0}(z)\sqrt{1-|w_0|^2}.
\end{split}
\]
It holds that
\begin{equation}
\label{xiHz0}
P_0H({w_0})=0
\end{equation}
and
\begin{equation}
\label{orthodecomp}
[F,F]=[H_0,H_0]+[H,H].
\end{equation}
\end{lemma}
\begin{proof}
First we have \eqref{xiHz0} since
\[
P_0H(w_0)=P_0F(w_0)-P_0F(w_0)e_{w_0}(w_0)\sqrt{1-|w_0|^2}=0.
\]
Using \eqref{xiHz0} we have
\[
[H,P_0F(w_0)e_{w_0}(z)\sqrt{1-|w_0|^2}]=F(w_0)^*P_0H(w_0)(1-|w_0|^2)=0.
\]
So, $[H,H_0]=0$ and
\[
[F,F]=[H_0+H,H_0+H]=[H_0,H_0]+[H,H].
\]
\end{proof}

To proceed and take care of the condition \eqref{xiHz0} (that is, in the scalar case, to divide by a Blaschke factor) we first need to define matrix-valued Blaschke factors. This is done in the next section.

\section{Matrix-valued Blaschke factors}
\label{sec3}
\setcounter{equation}{0}
Matrix-valued Blaschke factors originate with the work of Potapov \cite{pootapov} and can be defined (up to right multiplicative constant) as
\begin{equation}
B_{w_0,{P_0}}(z)=I_p-P_0+P_0b_{w_0}(z),
\label{potapovvv}
\end{equation}
where for $w_0\in\mathbb D,$ $b_{w_0}$ is defined as in (\ref{Szego and Moebius}), and $P_0\in\mathbb C^{p\times p}$ is any orthogonal projection. The {\sl degree} ${\rm deg}\, B_{w_0,P_0}$
of the Blaschke factor is by definition the rank of the projection $P_0$. When considering infinite products, it will be more convenient to consider for $w_0\not=0$ the Blaschke factor
\begin{equation}
\label{normalized}
\mathcal B_{w_0,{P_0}}(z)=I_p-P_0+P_0\frac{|w_0|}{w_0}\frac{w_0-z}{1-z\overline{w_0}}.
\end{equation}
Note that
\begin{equation}
\label{b-1}
B_{{w_0},P_0}^{-1}(z)=I_p-P_0+P_0\frac{1}{b_{w_0}(z)},
\end{equation}
and so
\[
\mathcal B_{w_0,{P_0}}(z)= B_{w_0,{P_0}}(z) U\quad\text{with}\quad U=I_p-P_0-\frac{|w_0|}{w_0}P_0.
\]

In \eqref{degB} in the following proposition, ${\rm deg}$ refers to the McMillan degree of a matrix-valued rational function. We refer e.g. to
\cite{bgk1} for the definition and properties of the McMillan degree and to \cite{Dym} for further information on matrix-valued Blaschke products. We also note that \eqref{rkhs} is a special case of \eqref{CAStein}, and that the proposition can be viewed as a special case of
Proposition \ref{tmbl}.

\begin{proposition}
\label{milano}
Let $B_{w_0,P_0}$ be defined by \eqref{potapovvv}. Then
\begin{equation}
K_{B_{w_0,P_0}}(z,w)\stackrel{\rm def.}{=}\frac{I_p-B_{{w_0},P_0}(z)B_{{w_0},P_0}(w)^*}{1-z\overline{w}}=\frac{(1-|w_0|^2)}{(1-z\overline{w_0})(1-w_0\overline{w})}P_0.
\label{rkhs}
\end{equation}
and
\[
\mathbf H_2^{p\times q}\ominus B_{w_0,P_0}\mathbf H_2^{p\times q}=\left\{\frac{P_0V}{1-z\overline{w_0}},\,\, V\in\mathbb C^{p\times q}\right\}
\]
is the reproducing kernel Hilbert space with reproducing kernel $K_{B_{w_0,P_0}}(z,w)$ meaning that  the function $z\mapsto
K_{B_{w_0,P_0}}(z,w)X$ belongs to $\mathbf H_2^{p\times q}\ominus B_{w_0,P_0}\mathbf H_2^{p\times q}$ for every $X\in\mathbb C^{p\times q}$ and
\[
[F(\cdot),K_{B_{w_0,P_0}}(\cdot,w)X]=[P_0F(w_0),X].
\]
Finally (and with $q=1$)
\begin{equation}
\label{degB}
{\rm deg}\, B_{w_0,P_0}={\rm dim}\,\mathbf H_2^{p}\ominus B_{w_0,P_0}\mathbf H_2^{p}.
\end{equation}
\end{proposition}
\begin{proof} We put the proof for completeness. In the proof we write $B(z)$ rather than $B_{w_0,P_0}$ to ease the notation. We have
Since $P_0(I_p-P_0)=0$ we have
\[
B(z)B(w)^*=I_p-P_0+P_0b_{w_0}(z)\overline{b_{w_0}(w)},
\]
and so
\[
I_p-B(z)B(w)^*=P_0(1-b_{w_0}(z)\overline{b_{w_0}(w)}).
\]
Equation \eqref{rkhs} follows in since
\[
\frac{1-b_{w_0}(z)\overline{b_{w_0}(w)}}{1-z\overline{w}}=\frac{(1-|w_0|^2)}{(1-z\overline{w_0})(1-w_0\overline{w})}.
\]
It follows that the $\mathbb C^{p\times p}$-valued function $K_{B_{w_0,P_0}}(z,w)$ is positive definite in the open unit disk, and that the
associated reproducing kernel Hilbert space $\mathcal H(K_{B{w_0,P_0}})$ of $\mathbb C^{p\times q}$-valued functions is exactly the set of functions of the form
$z\mapsto \frac{P_0V}{1-z\overline{w_0}}$ when  $V$ varies in $\mathbb C^{p\times q}$.
Equation \eqref{rkhs} also implies that the space $\mathcal H(K_{B{w_0,P_0}})$ is isometrically included in $\mathbf H_2^{p\times q}$.
That
\[\mathcal H(K_{B{w_0,P_0}})=\mathbf H_2^{p\times q}\ominus B_{w_0,P_0}\mathbf H_2^{p\times q}\]
follows then from the kernel decomposition
\[
\frac{I_p}{1-z\overline{w}}=\frac{I_p-B(z)B(w)^*}{1-z\overline{w}}\frac{B(z)B(w)^*}{1-z\overline{w}}.
\]
The last claim follows from the identification of the McMillan degree of an unitary rational function and the dimension of its associated reproducing kernel space; see for instance \cite{ad3,ag2} for the latter.
\end{proof}

We note that in Proposition \ref{milano} one can replace $B_{w_0,P_0}$ by ${\mathcal B}_{w_0,P_0}$. It holds that
\[
K_{B_{w_0,P_0}}(z,w)=K_{{\mathcal B}_{w_0,P_0}}(z,w).
\]

\begin{lemma} Let $H\in\mathbf H_2^{p\times q}$ be such that $P_0H(w_0)=0_{p\times q}$. Then
\[
G=B_{{w_0},P_0}^{-1}H\in\mathbf H_2^{p\times q}
\]
and
\begin{equation}
\label{paris}
[H,H]=[G,G].
\end{equation}
\end{lemma}
\begin{proof}
\eqref{b-1}
\[
B_{{w_0},P_0}^{-1}(z)=I_p-P_0+P_0\frac{1}{b_{w_0}(z)}.
\]
Write $P_0H(z)=(z-w_0)R(z)$, where $R$ is $\mathbb C^{k\times q}$-valued and analytic in a neighborhood of the origin. Using
\eqref{b-1} we have for $z\not=w_0$
\[
B_{{w_0},P_0}^{-1}(z)H(z)=P_0H(z)\frac{1}{b_{w_0}(z)}=R(z)(1-z\overline{w_0}),
\]
and the point $w_0$ is a removable singularity of $P_0 H$. Hence,
$B_{{w_0},P_0}^{-1}(z)H(z)$ has a removable singularity at $w_0$. Furthermore, since $B_{{w_0},P_0}(e^{it})^*B_{{w_0},P_0}(e^{it})=I_p$, and using \eqref{innerprod2}, we have \eqref{paris}.
\end{proof}

\section{Backward-shift invariant subspaces}
\label{sec2}
\setcounter{equation}{0}
We define for $\alpha\in\mathbb C$ the resolvent-like operator
\[
R_\alpha f(z)=\begin{cases}\,\, \dfrac{f(z)-f(\alpha)}{z-\alpha},\,\,\, z\not=\alpha,\\
\,\, f^\prime(\alpha),\,\,\,\,\,\,\,\quad\hspace{0.6cm} z=\alpha,\end{cases}
\]
where the (possibly vector-valued) function $f$ is analytic in a neighborhood of $\alpha$.\\

A finite dimensional space $\mathcal M$ of $\mathbb C^{p\times q}$-valued functions analytic in a neighborhood of the origin is $R_0$-invariant if and only if there exists a pair of matrices $(C,A)\in\mathbb C^{p\times N}\times\mathbb C^{N\times N}$ which is observable, meaning $\cap_{u=0}^\infty \ker CA^u=\left\{0\right\}$ and
\[
\mathcal M=\left\{F(z)=C(I_N-zA)^{-1}X,\quad X\in\mathbb C^{N\times q}\right\}.
\]

The following proposition is a particular case of the Beurling-Lax theorem in the finite dimensional setting.

\begin{proposition}
\label{tmbl}
Let $(C,A)\in\mathbb C^{p\times N}\times\mathbb C^{N\times N}$ be an observable pair of matrices, and let $\mathcal M$ denote the span of the functions of the
form $F(z)=C(I_N-zA)^{-1}X$, where $X$ runs through $\mathbb C^{N\times q}$. Then $\mathcal M\subset \mathbf H_2^{p\times q}$ if and only if $\rho(A)<1$.
When this is the case, we have $\mathcal M=\mathbf H_2^{p\times q}\ominus B\mathbf H_2^{p\times q}$, that is,
\[
\mathcal M^{\perp}=B\mathbf H_2^{p\times q},
\]
%
where $B$ is a finite Blaschke product, defined  up to a unitary right constant, by the formula
\begin{equation}
\label{B}
B(z)=I_p-(1-z)C(I_N-zA)^{-1}\mathbf P^{-1}(I_N-A)^{-*}C^*,
\end{equation}
with
\begin{equation}
\label{Pstein}
\mathbf P=\sum_{u=0}^\infty A^{*u}C^*CA^u.
\end{equation}
\end{proposition}

\begin{proof} The first claim follows from the series expansion
\[
C(I_N-zA)=\sum_{u=0}^\infty z^uCA^u,
\]
and from the observability of the pair $(C,A)$.\smallskip

To prove the second claim we remark that \eqref{Pstein} indeed converges since $\rho(A)<1$ and that the matrix $P$ is the unique solution of the Stein equation
\begin{equation}
\label{stein}
\mathbf{P}-A^*\mathbf{P}A=C^*C.
\end{equation}
The second claim follows then from the identity
\begin{equation}
\label{CAStein}
C(I_N-zA)^{-1}\mathbf{P}^{-1}(I_N-wA)^{-*}C^*=\frac{I_p-B(z)B(w)^*}{1-z\overline{w}},
\end{equation}
which is proved by a direct computation, taking into account \eqref{stein}.
\end{proof}

Using the above theorem, or using state space methods, one can prove that a finite Blasckhe product is a finite product of degree one Blaschke factors. This result originates with the
work of Potapov \cite{pootapov}; see e.g. \cite{ag2} for a proof.\smallskip

Note that $q=1$ in the next proposition.

\begin{proposition}
Let $B$ be a $\mathbb C^{p\times p}$-valued Blaschke product. There is a one-to-one correspondence between factorizations $B=B_1B_2$
of $B$ into two Blaschke products (up to a right multiplicative unitary constant $U$ for $ B_1$ and the corresponding left multiplicative constant $U^{-1}$ for $B_2$) and
$R_0$-invariant subspaces of $\mathbf H_2^{p}\ominus B\mathbf H_2^p$.
\end{proposition}

The following proposition uses Beurling-Lax theorem (see \cite{laxphi}). In the statement a $\mathbb C^{p\times \ell}$-valued inner function is an analytic
 $\mathbb C^{p\times \ell}$-valued function $\Theta$ such that the operator of multiplication by $\Theta$ is an isometry from
$ \mathbf H_2^{\ell\times q}$ into $\mathbf H_2^{p\times q}$
\begin{proposition}
\label{iowa-city}
Let $F\in\mathbf H_2^{p\times q}$ and assume that the closed linear span $\mathcal M(F)$ of the functions
\[
R_0^nFX,\quad n=0,1,2,\ldots\,\,\, and\,\,\, X\in\mathbb C^{q\times q}
\]
is strictly included in $\mathbf H_2^{p\times q}$. Then there exists a $\mathbb C^{p\times \ell}$-valued inner function $\Theta$ such that
\begin{equation}
\mathcal M(F)=\mathbf H_2^{p\times q}\ominus\Theta \mathbf H_2^{\ell\times q}.
\end{equation}
\end{proposition}
\begin{proof}
The space $\mathcal M(F)$ is $R_0$ invariant, and so its
orthogonal complement $(\mathcal M(F))^{\perp}$ is invariant by multiplication by $z$. The result follows then from the Beurling-Lax theorem.
\end{proof}

Note that $\Theta$ need not be square; for instance, if $p=2$, we can have
\[
\Theta(z)=\begin{pmatrix}0\\ b(z)\end{pmatrix},
\]
where $b$ is a Blaschke product. Then,
\[
\mathcal M(F)=\left\{\begin{pmatrix}f\\g\end{pmatrix}\,\,,\,\,f\in\mathbf H_2(\mathbb D)\,\,\text{and}\,\, g\in\mathbf H_2(\mathbb D)\ominus b\mathbf H_2(\mathbb D)\right\}.
\]
Still for $p=2$, the case
\[
\Theta(z)=\frac{1}{\sqrt{2}}\begin{pmatrix}1\\ b(z)\end{pmatrix},
\]
where
\[
b(z)=\prod_{n=1}^N\frac{z-w_n}{1-z\overline{w_n}}=
c_0+\sum_{n=1}^N\frac{c_n}{1-z\overline{w_n}}\quad\text{for uniquely defined $c_0,\ldots, c_N\in\mathbb C$},
\]
when the zeros of the Blaschke product are all
different from $0$ and simple, leads to
\[
\mathcal M(F)=\left\{ \begin{pmatrix}
\overline{c_0}f(z)+\sum_{n=1}^N
\overline{c_n}\frac{zg(z)-w_ng(w_n)}{z-w_n}
 \\ g(z) \end{pmatrix},\, g\in\mathbf H_2\right\}
\]
where we have used (for instance) \cite[Exercise 8.3.1]{CAPB2} to compute the first component.\\

These examples suggest a classification of the functions $F\in\mathbf H_2^{p\times q}$ depending on the value $\ell$ and the precise structure of $\Theta$.
\section{The algorithm}
\setcounter{equation}{0}
For any $w_0$ in the unit disc and any projection $P_0,$ there holds the orthogonal decomposition
\[
\mathbf H_2^{p\times q}=\left(\mathbf H_2^{p\times q}\ominus B_{w_0,P_0}\mathbf H_2^{p\times q}\right)\oplus B_{w_0,P_0}
\mathbf H_2^{p\times q},
\]
as is explained in the following lemma.
\begin{lemma}
For any given $w_0$ and $P_0$ formula
\eqref{bastille}  can be rewritten in a unique way as an orthogonal sum (orthogonal also with respect to the $[\cdot,\cdot]$ form)
\begin{equation}
\label{bastille222}
F(z)=M_0e_{w_0}(z)\sqrt{1-|w_0|^2}+B_{{w_0},P_0}(z)F_1(z),
\end{equation}
where $M_0\in\mathbb C^{p\times q}$ and $F_1\in\mathbf H_2^{p\times q}$. We have
$M_0e_{w_0}\sqrt{1-|w_0|^2}\in\left(\mathbf H_2^{p\times q}\ominus B_{w_0,P_0}\mathbf H_2^{p\times q}\right)$ and
$B_{{w_0},P_0}F_1\in B_{w_0,P_0}
\mathbf H_2^{p\times q}$. Finally,
\begin{equation}
\label{opera}
[F,F]=(1-|w_0|^2)[P_0F(w_0), F(w_0)]+[F_1,F_1].
\end{equation}
\end{lemma}

\begin{proof}
We have
\begin{equation}
F(z)=M_0e_{w_0}(z)\sqrt{1-|w_0|^2}+B_{{w_0},P_0}(z)F_1(z),
\end{equation}
where $M_0=P_0F(w_0)$ and $F_1=B_{{w_0},P_0}^{-1}\left(F-P_0F(w_0)e_{w_0}\sqrt{1-|w_0|^2}\right)\in\mathbf H_2^{p\times q}$. By Lemma \ref{milano},
\[
P_0F(w_0)e_{w_0}\sqrt{1-|w_0|^2}\in\mathbf H_2^{p\times q}\ominus B_{w_0,P_0}\mathbf H_2^{p\times q}.
\]
Furthermore,
\begin{equation}
[P_0F(w_0)e_{w_0}\sqrt{1-|w_0|^2},B_{{w_0},P_0}(z)F_1]=0_{q\times q}
\end{equation}
and so \eqref{opera} holds.
\end{proof}
Note that the decomposition \eqref{bastille} is non-trivial if and only if $F\not\equiv 0_{p\times q}$.\smallskip

Assume that in \eqref{bastille222} $F_1\not\equiv0$. We can then reiterate and, after fixing $k_1\in\left\{1,\ldots,p\right\}$,  get a decomposition of the form \eqref{bastille222} for $F_2$,
\begin{equation}
\label{bastille22}
F_1(z)=P_1F(w_1)e_{w_1}(z)\sqrt{1-|w_1|^2}+B_{{w_1},P_1}(z)F_2(z),
\end{equation}
where $w_1$ is any complex number in the disc, and $P_1$ is any orthogonal projection of rank $k_1$.
Thus $F$ admits the orthogonal (also with respect to the $[\cdot,\cdot]$ form) decomposition (with $M_1=P_1F(w_1)$)
\begin{equation}
\begin{split}
F(z)&=M_0e_{w_0}(z)\sqrt{1-|w_0|^2}+\\
&\hspace{5mm}+B_{w_0,P_0}(z)M_1e_{w_1}(z)\sqrt{1-|w_1|^2}+
B_{w_0,P_0}(z)B_{w_1,P_1}(z)F_2(z)
\end{split}
\label{algo2}
\end{equation}
along the decomposition
\[
\mathbf H_2^{p\times q}=\left(\mathbf H_2^{p\times q}\ominus B_{w_0,P_0}\mathbf H_2^{p\times q}\right)\oplus B_{w_0,P_0}\left(
\mathbf H_2^{p\times q}\ominus B_{w_1,P_1}\mathbf H_2^{p\times q}\right)\oplus
B_{w_0,P_0} B_{w_1,P_1}\mathbf H_2^{p\times q}
\]
of $\mathbf H_2^{p\times q}$. Note that
\[
\left(\mathbf H_2^{p\times q}\ominus B_{w_0,P_0}\mathbf H_2^{p\times q}\right)\oplus B_{w_0,P_0}\left(
\mathbf H_2^{p\times q}\ominus B_{w_1,P_1}\mathbf H_2^{p\times q}\right)
=\mathbf H_2^{p\times q}\ominus B_{w_0,P_0}B_{w_1,P_1}\mathbf H_2^{p\times q}.
\]

Iterating the algorithm we get a family $F_0=F,F_1,F_2,\ldots$ of functions in $\mathbf H_2^{p\times q}$ such that
\begin{equation}
F_{k}(z)=\left(B_{w_{k-1},P_{k-1}}(z)\right)^{-1}\left(F_{k-1}(z)-
M_{k-1}e_{w_{k-1}}(z)\sqrt{1-|w_{k-1}|^2}\right),\quad k=1,2,\ldots
\label{reduced2}
\end{equation}
where at each stage one takes a projection such that $P_kF_k\not\equiv 0$. If  there is no such projection it means that the algorithm ends at the given step.\\

The function $F_{k}$ is called the $k$-th reduced remainder  and is the matrix-valued analogue of  \eqref{red0}.
Let
\begin{equation}
\widetilde{B_0}(z)=P_0e_0(z)\quad\text{and}\quad \widetilde{B_k}(z)=P_ke_k(z)\stackrel{\curvearrowleft}{\prod_{u=0}^{k-1}} B_{w_u,P_u},\quad k=1,2,\ldots
\end{equation}
We have
\[
F(z)=\sum_{k=0}^{N} M_k\widetilde{B_k}(z)+{ B_{w_N,P_N}}(z)F_{N+1}(z).
\]

\begin{proposition}\label{notwith}
If ${ B_{w_N,P_N}}(z)F_{N+1}(z)=0,$ then the algorithm ends up after $N$ steps.  In such case $F$ is rational.
\end{proposition}

\begin{proof}
Indeed, if the algorithm finishes after a finite number of steps, there is a finite Blaschke product $B$ such that
$F\in\mathbf H_2^{p\times q}\ominus B\mathbf H_2^{p\times q}$, and the elements of the latter space are rational functions.
\end{proof}

If our selections of $w_k$ and $P_k$ follow the maximum selection principle (that is, because of the
choices of the point and the projection at each stage) we have the following result, which is the matrix-version of \cite[Theorem 2.2]{QWa1}.

\begin{theorem}
Suppose that at each step one selects $w_k$ and $P_k$ according to the maximum selection principle.
Then, the algorithm \eqref{reduced2} converges, meaning that
\[
\lim_{N\rightarrow N_0}{\rm Tr}\,[F(z)-\sum_{k=0}^{N} M_k\widetilde{B_k}(z),F(z)-\sum_{k=0}^{N} M_k\widetilde{B_k}(z)]=0,
\]
where $N_0$ can be finite or infinite.
In particular,
\begin{equation}
[F,F]=\sum_{k=0}^{N_0}[M_k,M_k],
\end{equation}
where $M_k=P_kF_k(w_k)$, $k=0,1,\ldots$, and
\begin{equation}
\label{fayetteville}
{\rm Tr}\,[F,F]=\sum_{k=0}^{N_0}{\rm Tr}\,[M_k,M_k].
\end{equation}
\end{theorem}

\begin{proof} The proof follows the proof for the scalar case presented in \cite[Theorem 2.2]{QWa1}. Before proceeding, it is important to recall
that the maximum \eqref{toronto} is computed on all projections of given rank and all points in $\mathbb D$.
The case $N_0 < \infty$ means that the algorithm ceases after a finite number of steps. We then obtain a decomposition of $F$ into a sum of finite entries, and $F$ is rational. We now suppose that $N_0=\infty.$ Let
\[
G=F-\sum_{k=0}^{\infty} M_k\widetilde{B_k}\not\equiv 0.
\]
We proceed in a number of steps to get a contradiction.\\

STEP 1: {\sl It holds that
\[
[F,F]=\sum_{k=0}^N[M_k,M_k]+[F_{N+1},F_{N+1}]
\]
with
\begin{equation}
M_k=[F,\widetilde{B_k}].
\label{ng}
\end{equation}
}
This follows from the unitarity of the Blaschke factors $ B_{w_u,P_u}$ on the unit circle that
\[
[ \widetilde{B_k}(z), \widetilde{B_\ell}(z)]=\begin{cases}\,0_{p\times p}\quad \text{if}\,\, k,\\
                                             \,P_k\,\quad\hspace{2.9mm}\text{if}\,\,k=\ell,\end{cases}
\]
and the claim in the step follows.\\

STEP 2: {\sl Let $R_k=F-\sum_{u=0}^{k-1}M_u\widetilde{B_u}$. We have
\begin{equation}
[F,\widetilde{B_k}]=[R_k,\widetilde{B_k}]=[F_k,P_ke_k]
\label{toto}
\end{equation}
}

The first equality in \eqref{toto} follows from
\begin{equation}
\label{clo}
[\widetilde{B_k},\widetilde{B_u}]=0_{p\times p}\quad\text{for}\,\,\,u=0,\ldots, k-1.
\end{equation}
The second equality follows from
\[
[R_k,\widetilde{B_k}]=[\left(\stackrel{\curvearrowleft}{\prod_{u=0}^{k-1}} B_{w_u,P_u}\right)F_k,Pe_k\left(
\stackrel{\curvearrowleft}{\prod_{u=0}^{k-1}} B_{w_u,P_u}\right)],
\]
and from the unitarity of the factors $ B_{w_u,P_u}$ on the unit circle.\\

STEP 3: {\sl  There exist a projection $P$, which we assume of rank one, and a point
$b\in\mathbb D$  such that
\[
{\rm Tr}~[G,Pe_b]={\rm Tr}~(PG(b))\not=0.
\]
}
In view of \eqref{clo} the sum
$\sum_{k=0}^{\infty} M_k\widetilde{B_k}$ converges in $\mathbf H_2^{p\times q}$ and so $G$ is analytic in the open unit disk.
The claim in the step follows thenfrom the analyticity of $G$ inside the open unit disk.\\

STEP 4: {\sl In the notation of the previous step, we have
\begin{equation}
\label{qwe45}
\sqrt{1-|b|^2}|{\rm Tr}\,[PR_k(b)]|>\frac{|{\rm Tr}~[G,Pe_b]|}{2}.
\end{equation}
}\smallskip

In  view of \eqref{fayetteville}, and using the Cauchy-Schwarz inequality we see that there is $k_0\in\mathbb N$ such that for all $k\ge k_0$,
\[
|{\rm Tr}\,[\sum_{u=k}^\infty [M_v\widetilde{B_v},Pe_b]|<\frac{|{\rm Tr}~[G,Pe_b]|}{2}.
\]
Hence,
\[
\begin{split}
|{\rm Tr}\,[R_k,Pe_b]|+\frac{|{\rm Tr}~[G,Pe_b]|}{2}&>|{\rm Tr}\,[R_k,Pe_b]|+|{\rm Tr}\,[\sum_{u=k}^\infty [M_v\widetilde{B_v},Pe_b]|\\
&\ge |{\rm Tr}\,[G,Pe_b]|,
\end{split}
\]
so that $|{\rm Tr}\,[R_k,Pe_b]|>\frac{|{\rm Tr}~[G,Pe_b]|}{2}$. By the reproducing kernel property this inequality can be rewritten as
\eqref{qwe45}.\\

STEP 5: {\sl We conclude the proof.}\smallskip

By the Cauchy-Schwarz inequality, and since $ B_{w_n,P_n}(b)^*B_{w_n,P_n}(b)\le I_p$, we have
\[
|{\rm Tr}\,[PR_k(b)]|<({\rm Tr}\,P)^{1/2}({\rm Tr}\, PR_k(b)^*R_k(b)P)^{1/2}<({\rm Tr}\,P)^{1/2}({\rm Tr}\, PF_k(b)^*F_k(b)P)^{1/2},
\]
and so
\[
\sqrt{1-|b|^2}({\rm Tr}\,P)^{1/2}({\rm Tr}\, PF_k(b)^*F_k(b)P)^{1/2}>\frac{|{\rm Tr}~[G,Pe_b]|}{2}.
\]
Since $P$ has rank $1$, it has trace equal to $1$ and we have
\[
\sqrt{1-|b|^2}({\rm Tr}\, PF_k(b)^*F_k(b)P)^{1/2}>\frac{|{\rm Tr}~[G,Pe_b]|}{2}.
\]
Equation \eqref{fayetteville} implies that $\lim_{k\rightarrow\infty}M_k=0$. From \eqref{toto} and \eqref{ng}
and the Cauchy-Schwarz inequality we have $\lim_{k\rightarrow\infty}[F_k,P_ke_k]=0$, and so
\[
\lim_{k\rightarrow\infty} \sqrt{1-|a_k|^2}P_kF_k(a_k)=0_{p\times p},
\]
and so
\[
\lim_{k\rightarrow\infty} (1-|a_k|^2){\rm Tr}\,[P_kF_k(a_k),F_k(a_k)]=0,
\]
and hence a contradiction with the maximum selection condition \eqref{toronto}, since
the maximum \eqref{toronto} is computed on all projections of given rank and all points in $\mathbb D$.
\end{proof}

\begin{remark}{\rm In the above arguments one could also use normalized factors of the form \eqref{normalized}. It is needed to use them when one wishes the underlying Blaschke product to
converge. See Remark \ref{r56}.}
\label{r5}
\end{remark}

\begin{remark} If $F$ is a rational general function, and the maximum selection principle is used at each step, then the only possibility that the algorithm stops after a finite $N_0$ steps is the case when $N_0=1.$  In the case, the subspace $\mathcal M(F)$ of $\mathbf H_2^{p\times q}$ spanned by the functions $R_0^uFX$ where $u=0,1,\ldots$ and $X\in\mathbb C^{q\times q}$ is finite dimensional and $R_0$-invariant by construction. So it is of the form
$\mathbf H_2^{p\times q}\ominus B\mathbf H_2^{p\times q}$ for some finite Blaschke product $B$. This does not mean that a backward-shift algorithm is then performed inside this
space, and thus that it would end after a finite number of steps. In contrast, for a rational function, if we do not use the maximum selection principle but suitably select $w_k$ and $P_k,$ the algorithm can well stop after a finite $N_0$ steps, as concerned by Proposition \ref{notwith}.
\end{remark}

\begin{remark}
\label{r56}
{\rm
We now consider the case where we take normalized Blaschke factors (see Remark \ref{r5}).
When the algorithm does not end in a finite number of steps, two cases occur depending on whether the infinite matrix-valued
Blaschke product
\begin{equation}
\label{estonie}
\mathcal B(z)\stackrel{\rm def.}{=}\stackrel{\curvearrowleft}{\prod_{n=0}^\infty}\mathcal B_{w_n,P_n}(z)=\lim_{N\rightarrow\infty}\mathcal B_{w_N,P_N}(z)\cdots
\mathcal B_{w_1,P_1}(z)\mathcal B_{w_0,P_0}(z)
\end{equation}
converges or not. The first case can be achieved by requiring the  numbers $a_n$ to satisfy
$\sum_{n=0}^\infty (1-|a_n|)<\infty$ (see \cite{QWa1}). The infinite product \eqref{estonie} then converges for all $z\in\mathbb D$ (the proof is as in the scalar case (see for instance
\cite[TG IX.82]{bourbaki_tg} for infinite products in a normed algebra) and $F\in\mathbf H_2^{p\times q}\ominus\mathcal B\mathbf H_2^{p\times q}$.} The second case then occurs when $\sum_{n=0}^\infty (1-|a_n|)=\infty.$ In such case an infinite Blaschke product cannot be defined, but instead, the shift invariant space reduces to zero, and the backward shift invariant space coincides with the whole Hardy $H_2^{p\times q}$ space. The proof of this fact is based on the Beurling-Lax Theorem. In fact, if the backward shift invariant space does not
coincide with the Hardy $H_2^{p\times q}$ space, then its orthogonal complement is a non-trivial shift invariant space. By the Beurling-Lax Theorem the latter is of the form $\mathcal B\mathbf H_2^{p\times q},$ where $\mathcal B$ is the Blaschke product generated by the $w_k's$ and the $P_k's.$ But this contradicts with the condition  $\sum_{n=0}^\infty (1-|a_n|)=\infty.$
\end{remark}
\begin{remark}{\rm
Assume that the Blaschke product \eqref{estonie} converges. Then $F\in\mathbf H_2^{p\times q}\ominus\mathcal B\mathbf H_2^{p\times q}$. But this latter space is $R_0$ invariant, and so
\begin{equation}
\label{eq:inclusion}
\mathcal M(F)\subset\mathbf H_2^{p\times q}\ominus\mathcal B\mathbf H_2^{p\times q},
\end{equation}
and $F$ is not cyclic for $R_0$. Let $\Theta$ be the the $\mathbb C^{p\times \ell}$-valued function as in Theorem \ref{iowa-city}. The
isometric inclusion \eqref{eq:inclusion} implies that the kernel
\[
\frac{\Theta(z)\Theta(w)^*-\mathcal B(z)\mathcal B(w)^*}{1-z\overline{w}}=\frac{I_p-\mathcal B(z)\mathcal B(w)^*}{1-z\overline{w}}-\frac{I_p-\Theta(z)\Theta(w)^*}{1-z\overline{w}}
\]
is positive definite in $\mathbb D$. Leech's factorization theorem (see \cite{rr-univ}, \cite{MR3147402}, \cite{MR3147403}, \cite{add}) implies that there is a $\mathbb C^{\ell\times p}$-valued function $\Theta_1$ analytic and contractive in $\mathbb D$ and such that $\mathcal B(z)=\Theta(z)\Theta_1(z)$.
Since $\mathcal B$ takes unitary values almost everywhere on the unit circle it follows that $\ell=p$.}
\end{remark}

\section{The case of real signals}
\setcounter{equation}{0}
We  begin with two definitions. Let $A=(a_{jk})_{\substack{j=1,\ldots p\\ k=1,\ldots, q}}\in\mathbb C^{p\times q}$. We say that $A$ is real if the entries of
$A$ are real, that is $A=\overline{A}$, where $\overline{A}$ is the matrix with  $(j,k)$-entry $\overline{a_{jk}}$.\smallskip

A matrix-valued real signal of finite energy is a function of the form
\[
f(t)=A_0+\sum_{n=1}^\infty A_n\cos (nt)+B_n\sin (nt),
\]
where the matrices $A_n$ and $B_n$ belong to $\mathbb R^{p\times q}$ and such that (with $A^T$ denoting the transpose of the matrix $A$)
\[
{\rm Tr}\, (A_0^TA_0)+\sum_{n=1}^\infty{\rm Tr}\,(A_n^TA_n+B_n^TB_n)<\infty.
\]
Since
\[
f(t)=A_0+\sum_{n=1}^\infty A_n\frac{e^{int}+e^{-int}}{2}+B_n\frac{e^{int}-e^{-int}}{2i}
\]
we can rewrite $f(t)$
\[
f(t)=F(e^{it})=\sum_{n\in\mathbb Z}F_ne^{int},
\]
with
\[
F_n=\begin{cases} \, A_0\,,\quad\hspace{0.7cm}\text{if}\quad n=0,\\
\, \frac{A_n-iB_n}{2}\,,\quad\text{if}\quad n=1,2,\ldots\\
\, \frac{A_n+iB_n}{2}\,,\quad\text{if}\quad n=-1,-2,\ldots
\end{cases}
\]
Note that $F_{-n}=\overline{F_n}$. With these computations at hand we can state (in the statement $\mathbb T$ denotes  the unit circle):

\begin{proposition}
Let $F\in\mathbf L_2^{p\times q}(\mathbb T)$, with power series $F(e^{it})=\sum_{n\in\mathbb Z}F_ne^{int}$ and let $F_+(e^{it})=F_0+\sum_{n=1}^\infty F_ne^{int}$.
Then, $F_+\in\mathbf H_2^{p\times q}$ and
\begin{equation}
\label{totototo}
F(e^{it})=F_+(e^{it})+\overline{F_+(e^{it})}-F_0
\end{equation}
\end{proposition}
\begin{proof}
Let $F_-(e^{it})=\sum_{n=1}^\infty F_{-n}e^{-int}$. Since the Fourier coefficients are real we can write
\[
\begin{split}
\overline{F_+(e^{it})}&=F_0+\sum_{n=1}^\infty \overline{F_n}e^{-int}\\
&=F_0+\sum_{n=1}^\infty F_{-n}e^{-int}\\
&=F_0+F_-(e^{it}),
\end{split}
\]
and so \eqref{totototo} holds.
\end{proof}

The preceding result allows to approximate real matrix-valued  signals using the maximum selection principle algorithm presented in the previous sections.

\section{Concluding remarks}
\setcounter{equation}{0}
The method developed  in \cite{Qian_106,QWa1} is extended here to the matrix-valued case. The results have impacts to rational approximation and interpolation of matrix-valued functions. In a sequel to the present paper we may consider the case of the ball $\mathbb B_N$ of $\mathbb C^N$. Then the counterpart of Blaschke elementary factors exists
(see \cite{rudin-ball}), and Blaschke products can be defined; see \cite{akap1}. One has then to consider the Drury-Arveson space of the ball,
that is the reproducing kernel Hilbert space of functions analytic in $\mathbb B_N$ with reproducing kernel $\frac{1}{1-\sum_{j=1}^N
z_j\overline{w_j}}$ rather than the Hardy space of the ball, whose reproducing kernel is $\frac{1}{(1-\sum_{j=1}^N
z_j\overline{w_j})^N}$, see \cite{arveson-acta,MR80c:47010}. We note that in the later mentioned reproducing kernel Hilbert space, viz., the Hardy $\mathbf H_2$ space inside the ball, there exists the $H^\infty$-functional calculus of the radial Dirac operator $\sum_{k=1}^N z_k \frac{\partial}{\partial z_k},$ or, equivalently, the singular integral operator algebra generalizing the Hilbert transformation on the sphere (\cite{CQ}). More generally, one can consider complete Pick kernels, that is positive definite kernels whose inverse has one positive square, see  \cite{MR2154356,MR2394102,btv,GRS,quiggin}.

\bibliographystyle{plain}

\end{document}